\documentclass[12pt]{amsart}

\usepackage{cite}
\usepackage{geometry}
\usepackage{lmodern, microtype}
\usepackage{amsthm, amsmath, amssymb, amsfonts}
\usepackage{mathtools}
\mathtoolsset{showonlyrefs=true}

\usepackage{mathrsfs} 
\allowdisplaybreaks[3] 
\usepackage{hyperref}

\hypersetup{
    colorlinks=true,
    linkcolor=blue,
    filecolor=magenta,      
    urlcolor=cyan,
    citecolor=red,
}

\newtheorem{theorem}{Theorem}[section]
\newtheorem{proposition}[theorem]{Proposition}
\newtheorem{lemma}[theorem]{Lemma}
\newtheorem{corollary}[theorem]{Corollary}
\theoremstyle{definition}

\newtheorem{example}[theorem]{Example}
\theoremstyle{remark}
\newtheorem{remark}[theorem]{Remark}

\numberwithin{equation}{section}

\newcommand{\doilink}[1]{\href{https://doi.org/#1}{\nolinkurl{#1}}}

\newcommand{\F}{\mathbb F}
\newcommand{\A}{\mathcal A}

\newcommand{\Sq}{\operatorname{Sq}}
\newcommand{\im}{\operatorname{im}}
\newcommand{\Span}{\operatorname{span}}
\newcommand{\rank}{\operatorname{rank}}

\newcommand{\End}{\operatorname{End}}
\newcommand{\Ann}{\operatorname{Ann}}

\def\DD{D\kern-.7em\raise0.4ex\hbox{\char '55}\kern.33em}

\title[Binary minor certificates for $MO(2)$]{Binary Minor Certificates\\ for the $\mathcal A(1)$-hit problem of $MO(2)$}

\author{\DD\d{\u a}ng V\~o Ph\'uc}
\address{Department of Mathematics, FPT University, An Phu Thinh New Urban Area, Quy Nhon, Vietnam}
\email{dangphuc150488@gmail.com}
\thanks{ORCID: \url{https://orcid.org/0000-0002-6885-3996}}


\begin{document}

\begin{abstract}
Let $M=\widetilde H^*(MO(2);\F_2)$ and let $\A(1)$ be the subalgebra of the mod $2$ Steenrod algebra generated by $\Sq^1$ and $\Sq^2$. In this work, we determine the quotient $\F_2\otimes_{\A(1)}M$ in every internal degree $h\geq0$, where internal degree $h$ corresponds to cohomological degree $h+2$. To overcome the limitations of finite experimental extrapolation, our proof is strictly certifying: in each degree, we exhibit a determinant-one minor of the binary hit matrix and a complementary family of linear dual functionals annihilating every hit column, mathematically guaranteeing the exact rank for all degrees. This rigorously establishes the monomial generating families and module decomposition proposed in Shih's 2012 report. Furthermore, we derive a closed binomial formula for every Steenrod square from the length-two symmetric-polynomial model, correct a boundary omission in the $\A(0)$ calculation, compute the exact Hilbert series of the $\A(1)$-cohit, and prove an intrinsic direct-sum decomposition of $M$ into two stable $\A(1)$-submodules.
\end{abstract}

\maketitle

\section{Introduction}

The hit problem asks for a minimal set of generators of a graded module over the mod $2$ Steenrod algebra $\A$. For the polynomial algebra $P(s)=\F_2[x_1,\ldots,x_s]$, with each variable in degree one, the problem was initiated by Peterson and developed through work of Wood, Kameko, Singer, and many subsequent authors \cite{Peterson1989, Phuc2026Minors, Wood1989,Kameko1998,Singer1991,WalkerWood1}. Its finite-dimensional degree-$d$ form is a rank problem: one writes the positive Steenrod operations from lower degrees as the columns of a binary matrix, and the quotient by the column space is the cohit space in degree $d$.

The symmetric version replaces $P(s)$ by the invariant algebra $S(s):=P(s)^{\mathfrak S_s}$. Janfada and Wood \cite{JanfadaWood2002} proved the symmetric analogue of the Peterson vanishing criterion and solved the two-variable case. They subsequently determined a minimal $\A$-generating set for $H^*(BO(3);\F_2)=S(3)$ \cite{JanfadaWood2003}. Janfada's subsequent papers developed the symmetric hit problem and related ordinary and symmetrized monomial criteria in three variables, including a special case of a symmetric-hit conjecture and instability refinements \cite{Janfada2008,Janfada2009,Janfada2011,Janfada2014}. Walker and Wood \cite{WalkerWood1,WalkerWood2} organized these results through the length decomposition of the algebra of symmetric functions, the up and down Kameko maps, and the dual symmetric Steenrod kernel. Singer \cite{Singer2008} applied the bigraded Steenrod algebra to the dual symmetric problem. Pengelley and Williams \cite{PengelleyWilliams2011,PengelleyWilliams2015} constructed a Kudo--Araki--May action on the dual symmetric algebras and subsequently proved a general sparseness theorem at all primes.

The Thom space $MO(2)$ \cite{Thom1954} provides a natural meeting point of the symmetric hit problem and cobordism. If $u$ is the Thom class of the universal real two-plane bundle, then the Thom isomorphism and Thom identity give $\widetilde H^*(MO(2);\F_2)=u\F_2[w_1,w_2]$, with $|u|=2$, and $\Sq(u)=u(1+w_1+w_2)$; see \cite{Thom1952,MilnorStasheff1974}. As observed by Walker and Wood \cite[Section~25.6]{WalkerWood2}, the cohomology of $MO(n)$ is identified with the symmetric polynomials divisible by all variables, which in rank two is precisely the length-two symmetric summand. Consequently, the full $\A$-cohit of $MO(2)$ is determined by the two-variable symmetric theorem of Janfada and Wood \cite{JanfadaWood2002}; see also \cite[Theorem~25.2.1]{WalkerWood2}. However, the problem over finite subalgebras $\A(r)$ remains distinct, as the natural surjection from the $\A(r)$-cohit to the full $\A$-cohit need not be injective. Finite-subalgebra hit phenomena have also been studied in other modules. Ault \cite{Ault2014} obtained Bott-type periodicity and Hilbert-series information for the $\A(1)$-indecomposables of tensor powers of the augmentation ideal in $H^*(\mathbb{R}P^\infty;\F_2)$. Ault's formulas concern the tensor powers themselves; the module studied here is the length-two symmetric submodule.

In an unpublished 2012 report \cite{Shih2012}, Shih investigated the $\A(r)$-hit problem for $MO(2)$ for $0 \leq r \leq 3$. His Lemma~4.2 proposes the $\A(1)$ generating families, and his Theorem~4.3 discusses the corresponding module decomposition. While that report provides useful experimental formulas for the operations $\Sq^1$, $\Sq^2$, $\Sq^4$, and $\Sq^8$, its global conclusions rely on empirical extrapolation rather than formal proof. Specifically, two mathematical gaps remain unresolved. First, the proposed exact list for $\A(0)$ omits a boundary class in degrees congruent to $2$ modulo $4$. Second, and more critically, the all-degree conclusion for $\A(1)$ is deduced from finite low-degree column reductions. The observation that hit matrix entries depend periodically on exponent residues does not logically guarantee that pivot structures or matrix ranks stabilize as the internal degree increases; the source ranges and endpoint rows must also be rigorously controlled. Furthermore, the arguments against splitting for higher subalgebras analyze only the overlaps of specific cyclic submodules generated by a chosen list, which does not exclude decompositions arising from alternative bases or nontrivial module idempotents.

The present article solves the $\A(1)$ problem completely and converts the proposed combinatorial pattern into a rigorous all-degree theorem. To overcome the limitations of the finite experimental extrapolation in \cite{Shih2012}, we construct an explicit degreewise primal--dual rank certificate for the $\A(1)$-hit matrices of the Thom module. The certificate consists of a determinant-one minor and a complementary family of linear forms in the annihilator of the hit space. Rather than tracing a full Gaussian elimination step-by-step, we establish exactness simultaneously from two sides: the minor bounds the hit matrix rank from below, while the linear forms bound the quotient dimension from below. The agreement of these two bounds mathematically guarantees the exact rank in every degree. The underlying rank criterion relies on elementary linear algebra; related minor bounds for polynomial hit matrices are considered in our recent work \cite{Phuc2026Minors}, while the dual Steenrod-kernel approach is developed in \cite[Chapter~9]{WalkerWood1} and \cite[Chapter~26]{WalkerWood2}. Furthermore, the structural module decomposition is verified directly via the stability of two intrinsic monomial spans, avoiding any reliance on prior generating-set assumptions.

The organization of this paper is as follows. Section~\ref{s2} identifies the Thom module with the length-two symmetric summand, derives the general binomial formula for the Steenrod action, and formulates the binary certificate principle. Section~\ref{s3} solves the $\A(0)$ problem and isolates the missing boundary family. Section~\ref{s4} performs the degreewise $\A(1)$ upper-bound reductions. Section~\ref{s5} constructs the determinant-one minors and dual annihilators, thereby proving exactness. Finally, Section~\ref{s6} derives the Hilbert series, establishes the direct-sum decomposition, and outlines the remaining extension questions for higher subalgebras.

\section{The symmetric model, Steenrod action, and binary certificates}\label{s2}

This section establishes the algebraic and linear-algebraic framework for the paper. We first pass from the Thom description of $MO(2)$ to its length-two symmetric-polynomial model, then derive a closed binomial formula for every Steenrod square in Thom coordinates, and finally formulate the primal--dual rank certificate used to determine the exact hit matrix ranks in all degrees.

Let $M=\widetilde H^*(MO(2);\F_2)$, and set
\[
S:=S(2)=H^*(BO(2);\F_2)=\F_2[w_1,w_2],
\qquad |w_1|=1,
\qquad |w_2|=2,
\]
which is the standard presentation of the mod $2$ cohomology of $BO(2)$ \cite{MilnorStasheff1974}. Let $u\in M$ be the Thom class of the universal real two-plane bundle. The Thom identity \cite{Thom1952,MilnorStasheff1974} is
\begin{equation}\label{eq:thom-total}
\Sq(u)=u(1+w_1+w_2).
\end{equation}
Thus $\Sq^1(u)=uw_1$, $\Sq^2(u)=uw_2$, and $\Sq^k(u)=0$ for $k>2$.

Under the splitting-principle identification $S\cong\F_2[x,y]^{\mathfrak S_2}$ with $|x|=|y|=1$, we have $w_1=x+y$ and $w_2=xy$. For $a\geq b\geq1$, let
\[
m(a,b)=
\begin{cases}
 x^ay^b+x^by^a,&a>b\geq1,\\
 x^ay^a,&a=b\geq1.
\end{cases}
\]
The following identification is the rank-two case of the Thom-space description in \cite[Section~25.6]{WalkerWood2}; the description by length is given in \cite[Section~26.1, Proposition~26.1.2(iv)]{WalkerWood2}.

\begin{proposition}\label{prop:thom-length-two}
With $M=\widetilde H^*(MO(2);\F_2)$ and $S=\F_2[w_1,w_2]$ as above, the map
\[
\Phi:M\longrightarrow w_2S,
\qquad
\Phi(uf)=w_2f,
\]
is an isomorphism of graded left $\A$-modules. Under the splitting-principle identification $S\cong\F_2[x,y]^{\mathfrak S_2}$, the ideal $w_2S$ corresponds to the length-two symmetric summand spanned by the monomial symmetric functions $m(a,b)$ with $a\geq b\geq1$.
\end{proposition}

Since the length decomposition is a direct sum of $\A$-modules, the two-variable symmetric cohit theorem of Janfada and Wood \cite{JanfadaWood2002}, in the form of \cite[Theorem~25.2.1]{WalkerWood2}, restricts to the length-two summand. The retained symmetrized spikes are precisely those whose two exponents are positive.

\begin{corollary}\label{cor:full-A}
Let $D\geq2$ be a cohomological degree. Then
\[
\dim_{\F_2}\bigl(\F_2\otimes_{\A}M\bigr)^D
=
\begin{cases}
1,&D=(2^a-1)+(2^b-1)\text{ for some }a,b\geq1,\\
0,&\text{otherwise}.
\end{cases}
\]
When the dimension is one, the class is represented under $\Phi$ by the corresponding length-two symmetrized spike.
\end{corollary}

In contrast to this sparse behavior over the full algebra $\A$, where the cohit dimension is at most $1$ in every degree, the $\A(1)$-cohit grows linearly in even internal degrees: as shown in Theorem~\ref{thm:A1-exact} (see Section \ref{s5}), $\dim Q_h^{(1)}M = \lfloor(h+2)/8\rfloor + 1$ for every even $h\geq 0$. Because both the hit matrix size and the quotient dimension grow without bound, finite low-degree checks cannot guarantee rank stabilization, which motivates the all-degree minor-and-annihilator certification developed below.

To pass from the global description to explicit degreewise matrices over finite subalgebras, we now express every Steenrod square in Thom coordinates, specialize the action to $\Sq^1$ and $\Sq^2$, and formulate the rank-certificate principle used throughout the paper.

For $n,m\geq0$, set
\[
v_{n,m}=u w_1^n w_2^m.
\]
We use the internal grading
\[
|v_{n,m}|_{\mathrm{int}}=n+2m.
\]
Thus $v_{n,m}$ has cohomological degree $n+2m+2$.  For $h\geq0$, define
\[
M_h=\Span_{\F_2}\{v_{n,m}:n+2m=h\}.
\]
The ordered basis of $M_h$ will be
\begin{equation}\label{eq:ehm}
e_m^{(h)}=v_{h-2m,m},
\qquad
0\leq m\leq N_h:=\left\lfloor\frac h2\right\rfloor.
\end{equation}

All binomial coefficients below are reduced modulo $2$. For a nonnegative upper entry $a$, we set $\binom ab=0$ when $b<0$ or $b>a$. Integer coefficients multiplying vectors are interpreted through $\mathbb Z\to\F_2$.

Under Proposition~\ref{prop:thom-length-two}, Shih's symmetric-polynomial formula \cite[Equation~(3.2)]{Shih2012} gives the following expression by replacing the exponent of $w_2$ by $m+1$ and reindexing the sum. Its underlying elementary-symmetric identities are instances of the Wu formula \cite[Proposition~25.5.1]{WalkerWood2}.

\begin{theorem}\label{thm:general-square}
Let $n,m,k$ be nonnegative integers. Then
\begin{equation}\label{eq:general-square}
\Sq^k(v_{n,m})
=
\sum_{t=0}^{\min\{m+1,\lfloor k/2\rfloor\}}
\binom{m+1}{t}
\binom{m+n+1-t}{k-2t}
 v_{n+k-2t,m+t}.
\end{equation}
\end{theorem}

\begin{proof}
Apply \cite[Equation~(3.2)]{Shih2012} to $\Phi(v_{n,m})=w_1^n w_2^{m+1}$. In that formula put $j=m+1+t$ and $i=n+k-2t$. The conditions $j\geq m+1$ and $i\geq n$, together with the first binomial coefficient, restrict the nonzero terms to
\[
0\leq t\leq\min\{m+1,\lfloor k/2\rfloor\}.
\]
For these terms the coefficient becomes
\[
\binom{m+1}{j-m-1}\binom{2m+n+2-j}{i-n}
=\binom{m+1}{t}\binom{m+n+1-t}{k-2t},
\]
and the output is $w_1^{n+k-2t}w_2^{m+1+t}$. Applying $\Phi^{-1}$ gives \eqref{eq:general-square}. Every upper binomial entry is nonnegative, and every output has internal degree
\[
(n+k-2t)+2(m+t)=n+2m+k.
\]
If a coefficient is nonzero, then $k-2t\leq n+m+1-t$ and $t\leq m+1$, whence $k\leq n+2m+2$. Thus the formula also gives zero above the cohomological instability bound.
\end{proof}

Specializing Theorem~\ref{thm:general-square} to $k=1,2$ and then expressing its inputs in the basis \eqref{eq:ehm} gives the two generator formulas needed for the hit matrices.

\begin{corollary}\label{cor:sq1-sq2}
Let $n,m\geq0$. Then
\begin{align}
\Sq^1(v_{n,m})
&=(n+m+1)v_{n+1,m},\label{eq:sq1-nm}\\
\Sq^2(v_{n,m})
&=\binom{n+m+1}{2}v_{n+2,m}+(m+1)v_{n,m+1}.\label{eq:sq2-nm}
\end{align}
For $h\geq1$ and $0\leq m\leq\lfloor(h-1)/2\rfloor$,
\begin{equation}\label{eq:sq1-target}
\Sq^1\bigl(e_m^{(h-1)}\bigr)=(h-m)e_m^{(h)}.
\end{equation}
For $h\geq2$ and $0\leq m\leq\lfloor(h-2)/2\rfloor$,
\begin{equation}\label{eq:sq2-target}
\Sq^2\bigl(e_m^{(h-2)}\bigr)
=\binom{h-m-1}{2}e_m^{(h)}+(m+1)e_{m+1}^{(h)}.
\end{equation}
\end{corollary}

\begin{proof}
For $k=1$, only $t=0$ occurs in \eqref{eq:general-square}, with coefficient $n+m+1$. For $k=2$, the terms $t=0,1$ have coefficients $\binom{n+m+1}{2}$ and $m+1$, respectively. This proves \eqref{eq:sq1-nm}--\eqref{eq:sq2-nm}. Substituting $n=h-1-2m$ in the first formula and $n=h-2-2m$ in the second gives \eqref{eq:sq1-target}--\eqref{eq:sq2-target}. The stated source ranges are exactly the conditions that these values of $n$ be nonnegative.
\end{proof}

\begin{remark}\label{rem:Shih-action}
The expanded $\Sq^2$ display in \cite[Section~4.2]{Shih2012} interchanges the coefficients of $v_{n+2,m}$ and $v_{n,m+1}$. At $n=m=0$ it would give $v_{2,0}$, whereas the Thom identity gives $\Sq^2(u)=v_{0,1}$. The expanded $\Sq^4$ display in its Section~4.3 also differs from \eqref{eq:general-square}: at $(n,m)=(2,0)$ it gives zero, while \eqref{eq:general-square} gives $\Sq^4(v_{2,0})=v_{4,1}$. These observations concern those expanded displays, not the basic Cartan identities or the proposed $\A(1)$ generating families.
\end{remark}

For $r\geq0$, let
\[
\A(r)=\langle\Sq^1,\Sq^2,\Sq^4,\ldots,\Sq^{2^r}\rangle
\]
denote the standard finite sub-Hopf algebra of the Steenrod algebra \cite[Proposition~12.4.5]{WalkerWood1}, and let $\A(r)^+$ be its augmentation ideal. We adopt the convention $M_d=0$ for $d<0$, and define
\[
H_h^{(r)}:=(\A(r)^+M)_h,
\qquad
Q_h^{(r)}M:=M_h/H_h^{(r)}.
\]

The following standard description of the hit space is used in \cite{Ault2014}. It follows from the generating set in \cite[Proposition~12.4.5]{WalkerWood1}: every nonempty product of these generators has a leftmost factor $\Sq^{2^j}$, so its image lies in the image of that generator. Taking the component of internal degree $h$ gives the stated equality.

\begin{lemma}\label{lem:generator-images}
Let $h,r\geq0$. Then
\[
H_h^{(r)}
=
\sum_{j=0}^{r}
\im\bigl(\Sq^{2^j}:M_{h-2^j}\longrightarrow M_h\bigr).
\]
\end{lemma}

The hit matrix in degree $h$ has rows indexed by \eqref{eq:ehm} and columns given by the generator images in Lemma~\ref{lem:generator-images}. If $U$ is a subspace of a finite-dimensional vector space $T$, write
\[
\Ann(U)=\{\lambda\in T^*: \lambda(u)=0\text{ for every }u\in U\}.
\]
The annihilator description is the finite-dimensional duality underlying the Steenrod kernel; see \cite[Definition~9.4.1, Proposition~9.4.2]{WalkerWood1} and \cite[Definition~26.1.5]{WalkerWood2}. The determinant of a $0\times0$ matrix is taken to be $1$. The following elementary linear-algebra criterion combines this dual description with a minor bound and includes the empty-minor case.

\begin{proposition}\label{prop:certificate-principle}
Let $N,k,q$ be nonnegative integers with $k\leq N$, let $T$ be an $N$-dimensional vector space over $\F_2$ with a fixed basis, and let $C$ be an $N\times q$ matrix whose columns span $U\subseteq T$. Let $g_1,\ldots,g_k\in T$. Suppose that $C$ has an $(N-k)\times(N-k)$ minor of determinant $1$, and that $\lambda_1,\ldots,\lambda_k\in\Ann(U)$ satisfy
\[
\lambda_i(g_j)=\delta_{ij}\qquad(1\leq i,j\leq k).
\]
Then
\[
\rank C=\dim U=N-k,\qquad \dim(T/U)=k.
\]
The classes of $g_1,\ldots,g_k$ form a basis of $T/U$, and $\lambda_1,\ldots,\lambda_k$ form a basis of $\Ann(U)$.
\end{proposition}

\begin{proof}
A linear relation among the selected columns restricts on the selected rows to a relation among the columns of an invertible square matrix. Its coefficients are therefore zero. Thus $\dim U\geq N-k$.

Define $L:T\to\F_2^k$ by
\[
L(x)=(\lambda_1(x),\ldots,\lambda_k(x)).
\]
The pairing hypothesis says that $L(g_j)$ is the $j$th standard basis vector, so $L$ is surjective. Rank--nullity gives $\dim\ker L=N-k$. Each $\lambda_i$ vanishes on $U$, and hence $U\subseteq\ker L$. Combining this inclusion with the preceding lower bound gives $U=\ker L$ and $\dim U=N-k$. Consequently $L$ induces an isomorphism $T/U\to\F_2^k$ taking $[g_j]$ to the $j$th standard basis vector.

The forms $\lambda_i$ are linearly independent: evaluating a relation $\sum_i c_i\lambda_i=0$ at $g_j$ gives $c_j=0$. Every form vanishing on $U$ factors uniquely through $T/U$, so $\Ann(U)$ is naturally isomorphic to $(T/U)^*$ and has dimension $k$. The independent forms $\lambda_i$ therefore form a basis. When $k=0$ the target $\F_2^0$ and both asserted bases are empty; the same dimension argument applies. When $N-k=0$, the first lower bound is $\dim U\geq0$, as required.
\end{proof}

\section{The \texorpdfstring{$\mathcal A(0)$}{A(0)} cohit and its boundary term}\label{s3}

Before analyzing the full $\A(1)$-action, we resolve the one-generator subalgebra $\A(0)$ in every internal degree, determining explicit monomial bases and exact dimensions for the quotient $Q_h^{(0)}M$. In particular, we isolate the boundary class $v_{0,h/2}$ in degrees $h\equiv 2\pmod 4$ that was omitted in \cite{Shih2012}, and account for its survival via the $\Sq^1$-Margolis homology of $M$.

Since $\A(0)$ is generated by $\Sq^1$, Lemma~\ref{lem:generator-images} gives
\[
H_h^{(0)}=\im\bigl(\Sq^1:M_{h-1}\longrightarrow M_h\bigr),
\qquad
Q_h^{(0)}M=M_h/H_h^{(0)}.
\]

\begin{theorem}\label{thm:A0}
Let $h\geq0$. If $h$ is odd, a basis of $Q_h^{(0)}M$ is represented by
\begin{equation}\label{eq:A0-odd}
\bigl\{v_{h-4a-2,\,2a+1}:a\geq0,\ h-4a-2\geq0\bigr\}.
\end{equation}
If $h$ is even, a basis is represented by
\begin{equation}\label{eq:A0-even-main}
\bigl\{v_{h-4a,\,2a}:a\geq0,\ h-4a\geq0\bigr\},
\end{equation}
together with the additional class $v_{0,h/2}$ when $h\equiv2\pmod4$.
\end{theorem}

\begin{proof}
If $h=0$, then $M_{-1}=0$ and $M_0=\Span_{\F_2}\{v_{0,0}\}$, which gives the stated basis. Assume $h\geq1$. For each existing source index, \eqref{eq:sq1-target} states that
\[
\Sq^1\bigl(e_m^{(h-1)}\bigr)=(h-m)e_m^{(h)}.
\]
No column contains more than one target basis vector.  Hence a target vector $e_m^{(h)}$ is hit if and only if the source $e_m^{(h-1)}$ exists and the scalar $h-m$ is odd.

Assume first that $h=2H+1$ is odd.  Both $M_h$ and $M_{h-1}$ have basis indices $0\leq m\leq H$.  Since
\[
h-m\equiv1-m\pmod2,
\]
the coefficient is one exactly when $m$ is even.  Therefore the non-hit indices are the odd integers
\[
m=2a+1,
\qquad
2a+1\leq H.
\]
For such an index,
\[
e_{2a+1}^{(h)}
=v_{h-2(2a+1),\,2a+1}
=v_{h-4a-2,\,2a+1},
\]
which gives \eqref{eq:A0-odd}.

Now assume that $h=2H$ is even.  The target indices are $0\leq m\leq H$, while the source indices are $0\leq m\leq H-1$.  For every source index,
\[
h-m\equiv m\pmod2,
\]
so the hit indices are precisely the odd integers not exceeding $H-1$.  Every even index is therefore non-hit.  Writing $m=2a$ gives the family \eqref{eq:A0-even-main}.

There remains the target endpoint $m=H$, which has no source with the same index.  If $H$ is even, then this endpoint already belongs to the family $m=2a$.  If $H$ is odd, then it is not in that family and must be added separately.  The condition that $H$ be odd is equivalent to $h=2H\equiv2\pmod4$, and the endpoint vector is
\[
e_H^{(h)}=v_{h-2H,H}=v_{0,h/2}.
\]
Since every $\Sq^1$ column is a scalar multiple of a single basis vector, the listed non-hit vectors are automatically linearly independent modulo the image.  They therefore form the asserted bases.
\end{proof}

Counting the representatives in Theorem~\ref{thm:A0} in each residue class modulo four gives the following dimensions.

\begin{corollary}\label{cor:A0-dim}
Let $h=4q+r$, where $q\geq0$ and $0\leq r<4$. Then
\[
\dim Q_h^{(0)}M=
\begin{cases}
q+1,&r=0,\\
q,&r=1,\\
q+2,&r=2,\\
q+1,&r=3.
\end{cases}
\]
\end{corollary}

\begin{proof}
The assertion follows by counting the indices in Theorem~\ref{thm:A0}.  For $h=4q$, the even indices from $0$ to $2q$ contribute $q+1$ classes.  For $h=4q+2$, the even indices from $0$ to $2q$ contribute $q+1$ classes and the odd endpoint contributes one more.  For $h=4q+1$, the odd indices from $1$ to $2q-1$ contribute $q$ classes.  For $h=4q+3$, the odd indices from $1$ to $2q+1$ contribute $q+1$ classes.
\end{proof}

\begin{remark}\label{rem:A0-boundary}
The first omitted case occurs at internal degree $h=2$.  The endpoint is $v_{0,1}=uw_2$.  The space $M_1$ is spanned by $v_{1,0}=uw_1$, and
\[
\Sq^1(uw_1)=\Sq^1(u)w_1+u\Sq^1(w_1)
=uw_1^2+uw_1^2=0.
\]
Thus $uw_2$ is not hit over $\A(0)$. More generally, for $h=4q+2$ with $q\geq0$, the $\Sq^1$-cycles have odd indices $1,3,\ldots,2q+1$, whereas the boundaries have odd indices $1,3,\ldots,2q-1$ (none when $q=0$). Since $(\Sq^1)^2=0$, the resulting Margolis homology is
\[
H_{4q+2}(M;\Sq^1)
:=\frac{\ker(\Sq^1:M_{4q+2}\to M_{4q+3})}
{\im(\Sq^1:M_{4q+1}\to M_{4q+2})}
=\F_2\{[v_{0,2q+1}]\}.
\]
Here $v_{0,2q+1}=e_{2q+1}^{(4q+2)}$. The outgoing square has rank $q+1$, so this unique homology class accounts for the additional dimension in $\dim Q_{4q+2}^{(0)}M=q+2$.
\end{remark}

\section{Degreewise reduction for \texorpdfstring{$\mathcal A(1)$}{A(1)}}\label{s4}

Passing from $\A(0)$ to $\A(1)$ adjoins the generator $\Sq^2$, whose columns contain at most two target basis vectors. Combining the $\Sq^1$-eliminations of Section~\ref{s3} with binary Lucas congruences and an endpoint analysis modulo $8$, we show that $Q_h^{(1)}M$ vanishes in every odd internal degree and is spanned in every even degree by the explicit candidate set $G_h$ of \cite[Lemma~4.2]{Shih2012}, yielding the upper bound $\dim Q_h^{(1)}M \leq \lfloor(h+2)/8\rfloor + 1$. Linear independence of these classes is established in Section~\ref{s5} via determinant-one minors and complementary dual annihilators.

By Lemma~\ref{lem:generator-images},
\begin{equation}\label{eq:A1-hit-space}
H_h^{(1)}
=
\im\bigl(\Sq^1:M_{h-1}\to M_h\bigr)
+
\im\bigl(\Sq^2:M_{h-2}\to M_h\bigr).
\end{equation}

\begin{proposition}\label{prop:A1-odd}
Let $h\geq1$ be odd. Then $Q_h^{(1)}M=0$.
\end{proposition}

\begin{proof}
Write $h=2H+1$. The target basis is $e_m^{(h)}$ for $0\leq m\leq H$. By Theorem~\ref{thm:A0}, every even-indexed target vector is hit by $\Sq^1$.

Let $m=2a+1$ be odd.  Since $2a+1\leq H$, the index $2a$ satisfies $2a\leq H-1$, so $e_{2a}^{(h-2)}$ exists.  Formula \eqref{eq:sq2-target} gives
\begin{align*}
\Sq^2\bigl(e_{2a}^{(h-2)}\bigr)
&=\binom{h-2a-1}{2}e_{2a}^{(h)}+(2a+1)e_{2a+1}^{(h)}\\
&=\binom{h-2a-1}{2}e_{2a}^{(h)}+e_{2a+1}^{(h)}.
\end{align*}
The first term on the right is already hit by $\Sq^1$.  Therefore $e_{2a+1}^{(h)}$ is hit modulo the $\Sq^1$ image, and hence it belongs to the sum \eqref{eq:A1-hit-space}.  Every target basis vector is hit, so the quotient is zero.
\end{proof}

For even $h$, define the candidate set $G_h$ from \cite[Lemma~4.2]{Shih2012} as follows.  When $h\equiv0\pmod4$, let
\begin{equation}\label{eq:Gh-zero}
G_h=
\bigl\{e_{2+4a}^{(h)}:a\geq0,\ 2+4a\leq h/2\bigr\}
\cup
\begin{cases}
\{e_{h/2}^{(h)}\},&h\equiv0\pmod8,\\
\emptyset,&h\equiv4\pmod8.
\end{cases}
\end{equation}
When $h\equiv2\pmod4$, let
\begin{equation}\label{eq:Gh-two}
G_h=
\bigl\{e_{4a}^{(h)}:a\geq0,\ 4a\leq h/2\bigr\}
\cup
\begin{cases}
\{e_{h/2}^{(h)}\},&h\equiv6\pmod8,\\
\emptyset,&h\equiv2\pmod8.
\end{cases}
\end{equation}

\begin{proposition}\label{prop:A1-even-spanning}
Let $h\geq0$ be even. Then the quotient $Q_h^{(1)}M$ is spanned by the classes represented by $G_h$.
\end{proposition}

\begin{proof}
For $h=0$, the source spaces in \eqref{eq:A1-hit-space} are zero and $G_0=\{e_0^{(0)}\}$, so the assertion holds. Suppose that $h=2H>0$. Theorem~\ref{thm:A0} identifies the $\Sq^1$-hit target vectors as exactly the odd-indexed $e_m^{(h)}$ with $m\leq H-1$.

Now consider an even index $m\leq H-1$.  Formula \eqref{eq:sq2-target} gives
\begin{equation}\label{eq:even-sq2-relation}
\Sq^2\bigl(e_m^{(h-2)}\bigr)
=a_m e_m^{(h)}+e_{m+1}^{(h)},
\qquad
a_m=\binom{h-m-1}{2}.
\end{equation}
The second term is odd-indexed.  If $m+1\leq H-1$, it is already hit by $\Sq^1$.  The parity of $a_m$ is determined by the binary form of Lucas's theorem \cite{Fine1947}.  Since $h$ and $m$ are even, the integer $h-m-1$ is odd.  The lower entry $2$ has binary expansion $10$, so
\[
\binom{h-m-1}{2}=1
\]
if and only if the $2$-bit of $h-m-1$ is one.  For an odd integer this is equivalent to
\[
h-m-1\equiv3\pmod4.
\]
Adding one gives
\[
h-m\equiv0\pmod4,
\]
and therefore
\begin{equation}\label{eq:am-residue}
a_m=1
\quad\Longleftrightarrow\quad
m\equiv h\pmod4.
\end{equation}
For every even $m\leq H-2$, equation \eqref{eq:even-sq2-relation} therefore kills $e_m^{(h)}$ modulo the $\Sq^1$ image when $m\equiv h\pmod4$.  The even indices that remain are exactly those satisfying
\[
m\equiv h+2\pmod4.
\]
If $h\equiv0\pmod4$, these are the indices $m=2+4a$.  If $h\equiv2\pmod4$, these are the indices $m=4a$.  This accounts for the principal families in \eqref{eq:Gh-zero} and \eqref{eq:Gh-two}.

It remains to analyze the endpoint $m=H$.  There are two cases, according to the parity of $H$.

Suppose first that $H$ is even.  Then $e_H^{(h)}$ has even index.  It has no $\Sq^1$ source because the source range ends at $H-1$.  It can occur as the second term of a $\Sq^2$ column only from source index $H-1$, but that coefficient is
\[
(H-1)+1=H=0\quad\text{in }\F_2.
\]
It cannot occur as the first term of a $\Sq^2$ column because source index $H$ does not exist.  Thus $e_H^{(h)}$ survives.  If $H\equiv2\pmod4$, it is already in the principal family $m=2+4a$; this is the case $h\equiv4\pmod8$.  If $H\equiv0\pmod4$, it is not in that family and must be added separately; this is the case $h\equiv0\pmod8$.

Suppose next that $H$ is odd.  Then the endpoint is odd-indexed and is not hit by $\Sq^1$ because it has no source.  The source index $H-1$ is even, and \eqref{eq:sq2-target} gives
\begin{equation}\label{eq:boundary-sq2}
\Sq^2\bigl(e_{H-1}^{(h-2)}\bigr)
=\binom H2 e_{H-1}^{(h)}+e_H^{(h)}.
\end{equation}
The same binary congruence \cite{Fine1947} gives
\[
\binom H2=
\begin{cases}
0,&H\equiv1\pmod4,\\
1,&H\equiv3\pmod4.
\end{cases}
\]
If $H\equiv1\pmod4$, equivalently $h\equiv2\pmod8$, equation \eqref{eq:boundary-sq2} hits $e_H^{(h)}$ directly, while $e_{H-1}^{(h)}$ is one of the principal survivors because $H-1\equiv0\pmod4$.

If $H\equiv3\pmod4$, equivalently $h\equiv6\pmod8$, equation \eqref{eq:boundary-sq2} is the relation
\[
e_{H-1}^{(h)}+e_H^{(h)}=0
\quad\text{in }Q_h^{(1)}M.
\]
Here $H-1\equiv2\pmod4$, so $e_{H-1}^{(h)}$ is not in the principal survivor family.  The relation leaves at most one class from the pair, and we choose $e_H^{(h)}$ as its representative.  This is the additional term in \eqref{eq:Gh-two}.  All target basis vectors not represented in $G_h$ have now been reduced to the hit space or to a linear combination of the listed representatives.  Hence $G_h$ spans the quotient.
\end{proof}

The preceding spanning result gives a numerical bound by counting the elements of $G_h$.

\begin{corollary}\label{cor:A1-upper}
Let $h\geq0$ be even. Then
\[
\dim Q_h^{(1)}M
\leq
\left\lfloor\frac{h+2}{8}\right\rfloor+1.
\]
\end{corollary}

\begin{proof}
It suffices to count $G_h$.  Write $H=h/2$.

If $H=4q$, then the principal indices are $2,6,\ldots,4q-2$, giving $q$ elements, with the progression empty when $q=0$. The endpoint adds one, for a total of $q+1$.  If $H=4q+2$, the principal indices are $2,6,\ldots,4q+2$, giving $q+1$ elements.  If $H=4q+1$, the principal indices are $0,4,\ldots,4q$, giving $q+1$ elements.  If $H=4q+3$, the principal indices are $0,4,\ldots,4q$, giving $q+1$ elements, and the endpoint adds one, for a total of $q+2$.  In all four cases the result equals
\[
\left\lfloor\frac{H+1}{4}\right\rfloor+1
=
\left\lfloor\frac{h+2}{8}\right\rfloor+1.
\]
\end{proof}

\section{Determinant-one minors and dual annihilators}\label{s5}

The degreewise reductions of Section~\ref{s4} isolate the candidate spanning set $G_h$ in each even internal degree. We now establish exact linear independence by pairing a complementary minor of size $\dim M_h - |G_h|$ with a dual family of annihilating linear forms. Using the binomial residue and endpoint calculations from Section~\ref{s4}, we arrange the selected pivot rows and columns into an upper unit block-triangular submatrix of determinant $1$, and construct functionals in $\Ann(H_h^{(1)})$ that pair as the Kronecker delta with $G_h$, including the two-term boundary functional $\varepsilon_{H-1}+\varepsilon_H$ in degrees $h\equiv 6\pmod 8$. Applying Proposition~\ref{prop:certificate-principle} yields the exact rank of the $\A(1)$-hit matrix and proves that $G_h$ represents a basis of $Q_h^{(1)}M$ for every even $h\geq 0$ (Theorem~\ref{thm:A1-exact}), which we illustrate explicitly at $h=6$ in Example~\ref{ex:h6}.

Let $\mathsf{B}_h^{(1)}$ denote the $\A(1)$ hit matrix in degree $h$, with rows indexed by $e_m^{(h)}$ and columns indexed by all vectors
\[
\Sq^1(e_m^{(h-1)})
\quad\text{and}\quad
\Sq^2(e_m^{(h-2)})
\]
for which the source exists.

\begin{theorem}\label{thm:A1-exact}
Let $h\geq0$. If $h$ is odd, the matrix $\mathsf{B}_h^{(1)}$ has full row rank and $Q_h^{(1)}M=0$.  If $h$ is even, the classes represented by $G_h$ form a basis of $Q_h^{(1)}M$. In particular,
\[
\dim Q_h^{(1)}M=
\begin{cases}
0,&h\text{ odd},\\
\left\lfloor\dfrac{h+2}{8}\right\rfloor+1,&h\text{ even}.
\end{cases}
\]
The matrix $\mathsf{B}_h^{(1)}$ has a minor of size $\dim M_h-\dim Q_h^{(1)}M$ and determinant $1$.
\end{theorem}

\begin{proof}
Suppose first that $h=2H+1$. Let $E$ and $O$ be, respectively, the even and odd indices in $\{0,\ldots,H\}$. Select the $\Sq^1$ columns with source indices in $E$ and the $\Sq^2$ columns with source indices $m-1$ for $m\in O$. These are the columns used in Proposition~\ref{prop:A1-odd}; their source ranges follow from Theorem~\ref{thm:A0} and the proof of that proposition. In the row order $E,O$ and the corresponding column order, their matrix is
\begin{equation}\label{eq:unit-block-minor}
\begin{pmatrix} I_{|E|}&A\\0&I_{|O|}\end{pmatrix}.
\end{equation}
The first columns are the even coordinate vectors, and each remaining column has coefficient $1$ at its associated odd row and no other odd component. Thus the determinant is $1$. There are $H+1$ selected columns, so this is a full-row-rank minor. The annihilator is zero in this case.

Now let $h=2H$ be even. For $h=0$, the hit matrix has one row and no columns, and $G_0=\{e_0^{(0)}\}$. The empty minor has determinant $1$, and the coordinate form $\varepsilon_0$ annihilates the zero hit space and takes value $1$ on $e_0^{(0)}$. Proposition~\ref{prop:certificate-principle} therefore applies. Assume $h>0$ for the remainder of the proof. We construct a minor using all target rows except those in $G_h$.  For every odd $m$ with $1\leq m\leq H-1$, select
\[
x_m=\Sq^1(e_m^{(h-1)})=e_m^{(h)}.
\]
For every even $m$ with $0\leq m\leq H-1$ and
\[
a_m=\binom{h-m-1}{2}=1,
\]
select
\[
y_m=\Sq^2(e_m^{(h-2)})=e_m^{(h)}+e_{m+1}^{(h)}.
\]
There is one additional selection when $H\equiv1\pmod4$.  In that case $a_{H-1}=0$, and we select
\[
y_{H-1}=\Sq^2(e_{H-1}^{(h-2)})=e_H^{(h)}.
\]
This extra column pivots the boundary row $e_H^{(h)}$, which is not in $G_h$ when $h\equiv2\pmod8$.

The selected pivot rows are exactly the complement of $G_h$, by \eqref{eq:am-residue} and the endpoint analysis in Proposition~\ref{prop:A1-even-spanning}. In particular, when $H\equiv3\pmod4$, the regular column $y_{H-1}=e_{H-1}^{(h)}+e_H^{(h)}$ pivots $e_{H-1}^{(h)}$ after the endpoint row is omitted.

Order the retained rows by listing first the odd rows pivoted by the $x_m$, then the rows pivoted by the selected $y_m$.  Order the columns in the same two groups. The restricted matrix again has the block-triangular form \eqref{eq:unit-block-minor}. Every $x_m$ is a coordinate vector in the first row group. Each regular $y_m$ has its distinct even pivot $e_m^{(h)}$ and one odd component. If $m+1<H$, that odd component lies in the first row group. If $m+1=H$, the regular selection forces $H\equiv3\pmod4$, and the endpoint row is omitted. The special boundary column, when present, is exactly $e_H^{(h)}$ and has no other component. Consequently the lower-right block is an identity matrix and the lower-left block is zero.  Thus the determinant of the selected minor is one.  If $g_h=|G_h|$, this proves
\begin{equation}\label{eq:rank-lower}
\rank \mathsf{B}_h^{(1)}\geq \dim M_h-g_h.
\end{equation}

We next construct $g_h$ dual annihilators.  Let $\varepsilon_j\in M_h^*$ denote the coordinate functional
\[
\varepsilon_j(e_m^{(h)})=\delta_{jm}.
\]
For every element $e_j^{(h)}\in G_h$, except the endpoint in the congruence class $h\equiv6\pmod8$, take $\lambda_j=\varepsilon_j$.  We show that $\varepsilon_j$ annihilates every hit column.

First consider a $\Sq^1$ column.  Such a column is a scalar multiple of $e_m^{(h)}$, and it is nonzero only when $m$ is odd.  Every ordinary index $j$ in $G_h$ is even, so $\varepsilon_j$ vanishes on every $\Sq^1$ column.

Now consider a $\Sq^2$ column
\[
\binom{h-m-1}{2}e_m^{(h)}+(m+1)e_{m+1}^{(h)}.
\]
An ordinary candidate index $j$ can occur as the first term only when $m=j$.  For every non-boundary principal candidate, \eqref{eq:am-residue} gives
\[
\binom{h-j-1}{2}=0.
\]
An ordinary endpoint with $H$ even has no source index $m=H$, so it cannot occur as a first term.  If $j=0$, it cannot occur as a second term because the putative source index is negative. If $j>0$, it can occur as a second term only when $m=j-1$. Since $j$ is even, the coefficient is
\[
m+1=j=0\quad\text{in }\F_2.
\]
Therefore every ordinary $\varepsilon_j$ vanishes on every $\Sq^2$ column.

It remains to handle $h\equiv6\pmod8$.  Then $H=h/2\equiv3\pmod4$, and $G_h$ contains the endpoint $e_H^{(h)}$.  Define
\[
\lambda_H=\varepsilon_{H-1}+\varepsilon_H.
\]
The row $e_{H-1}^{(h)}$ is even-indexed, and
\[
\Sq^1(e_{H-1}^{(h-1)})=(h-H+1)e_{H-1}^{(h)}=(H+1)e_{H-1}^{(h)}=0
\]
because $H+1$ is even.  The endpoint has no $\Sq^1$ source.  Thus $\lambda_H$ annihilates all $\Sq^1$ columns.

Among the $\Sq^2$ columns, the only column containing $e_H^{(h)}$ is the boundary column with source index $H-1$.  Since $H\equiv3\pmod4$, equation \eqref{eq:boundary-sq2} gives
\[
\Sq^2(e_{H-1}^{(h-2)})=e_{H-1}^{(h)}+e_H^{(h)},
\]
and $\lambda_H$ evaluates on this column as $1+1=0$.  A column containing $e_{H-1}^{(h)}$ as a second term would have source index $H-2$, but its second coefficient would be $H-1$, which is even.  No other column contains either coordinate.  Hence $\lambda_H$ annihilates every $\Sq^2$ column.

The ordinary coordinate functionals pair as the identity with their corresponding elements of $G_h$. In the exceptional boundary case, $e_{H-1}^{(h)}$ is not in $G_h$, so
\[
\lambda_H(e_H^{(h)})=1
\]
and $\lambda_H$ vanishes on every other candidate.  Thus the complete family of annihilators pairs as the identity with $G_h$. Combining \eqref{eq:rank-lower} with Proposition~\ref{prop:certificate-principle} gives
\[
\rank \mathsf{B}_h^{(1)}=\dim M_h-g_h,
\qquad
\dim Q_h^{(1)}M=g_h,
\]
and the classes represented by $G_h$ form a basis.  Corollary~\ref{cor:A1-upper} evaluates $g_h$ as the stated dimension.
\end{proof}

\begin{samepage}
\begin{example}\label{ex:h6}
For $h=6$ ($H=3$, $D=8$), use rows $e_0^{(6)},\ldots,e_3^{(6)}$ and, in order, the columns $\Sq^1(e_m^{(5)})$, $0\leq m\leq2$, followed by $\Sq^2(e_m^{(4)})$, $0\leq m\leq2$. Then
\[
\mathsf B_6^{(1)}=
\begin{pmatrix}
0&0&0&0&0&0\\
0&1&0&1&0&0\\
0&0&0&0&0&1\\
0&0&0&0&0&1
\end{pmatrix}.
\]
The columns $\Sq^1(e_1^{(5)})=e_1^{(6)}$ and $\Sq^2(e_2^{(4)})=e_2^{(6)}+e_3^{(6)}$ give the minor $I_2=\bigl(\begin{smallmatrix}1&0\\0&1\end{smallmatrix}\bigr)$ on the rows corresponding to $e_1^{(6)}$ and $e_2^{(6)}$. The set $G_6=\{e_0^{(6)},e_3^{(6)}\}$ pairs as the identity with $\lambda_0=\varepsilon_0$ and $\lambda_3=\varepsilon_2+\varepsilon_3$, both of which vanish on all six columns. Proposition~\ref{prop:certificate-principle} therefore gives $\rank\mathsf B_6^{(1)}=2$ and $\dim Q_6^{(1)}M=2$.
\end{example}
\end{samepage}

\section{Hilbert series and the \texorpdfstring{$\mathcal A(1)$}{A(1)} splitting}\label{s6}

The degreewise bases established in Section~\ref{s5} assemble into minimal homogeneous generating families for the entire $\A(1)$-module $M$. We first justify this passage from degreewise cohit bases to global module generators by induction on internal degree and derive the closed rational Hilbert series of $Q^{(1)}M$ (Proposition~\ref{prop:global-generators}). We then verify directly that the two residue monomial spans $M^{[0,1]}$ and $M^{[2,3]}$ are $\A(1)$-stable, yielding the intrinsic module splitting $M = M^{[0,1]} \oplus M^{[2,3]}$ of \cite[Theorem~4.3]{Shih2012} together with minimal generating sets for each summand (Theorem~\ref{thm:A1-splitting} and Corollary~\ref{cor:summand-generators}). Finally, Remark~\ref{rem:higher-subalgebras} outlines the scope of our certification and idempotent criteria for the higher subalgebras $\A(2)$ and $\A(3)$.

Write
\[
Q^{(1)}M=\bigoplus_{h\geq0}Q_h^{(1)}M,
\qquad
P_{Q^{(1)}M}(t)=\sum_{h\geq0}\dim_{\F_2}(Q_h^{(1)}M)t^h.
\]

Reindexing the degreewise bases in Theorem~\ref{thm:A1-exact} gives the following homogeneous families; their generating functions determine the Hilbert series.

\begin{proposition}\label{prop:global-generators}
Let $M=\widetilde H^*(MO(2);\F_2)$ with the $\A(1)$ action of Corollary~\ref{cor:sq1-sq2}. Then a minimal homogeneous generating set for $M$ is the disjoint union $\mathcal G_{2,3}\cup\mathcal G_{0,1}$, where
\begin{align*}
\mathcal G_{2,3}
&=
\{v_{4a,\,2+4b}:a,b\geq0\}
\cup
\{v_{0,\,3+4a}:a\geq0\},\\
\mathcal G_{0,1}
&=
\{v_{4a+2,\,4b}:a,b\geq0\}
\cup
\{v_{0,\,4a}:a\geq0\}.
\end{align*}
Consequently, with internal degree recorded by $t$,
\begin{equation}\label{eq:hilbert-series}
P_{Q^{(1)}M}(t)
=
\frac{t^2+t^4}{(1-t^4)(1-t^8)}
+
\frac{1+t^6}{1-t^8}.
\end{equation}
The corresponding series in cohomological degree is
\[
P^{\mathrm{coh}}_{Q^{(1)}M}(t)=t^2P_{Q^{(1)}M}(t).
\]
\end{proposition}

\begin{proof}
We first show that the displayed families reproduce the bases in Theorem~\ref{thm:A1-exact}, using $e_m^{(h)}=v_{h-2m,m}$.  A vector $v_{4a,2+4b}$ has internal degree
\[
4a+2(2+4b)=4a+4+8b,
\]
which is divisible by four.  Conversely, a principal basis vector in a degree $h\equiv0\pmod4$ has second exponent $m=2+4b$ and first exponent $h-2m$, which is a nonnegative multiple of four.  Hence it has the form $v_{4a,2+4b}$.  The extra endpoint in degree $h\equiv0\pmod8$ is $v_{0,4a}$.

A vector $v_{4a+2,4b}$ has internal degree
\[
4a+2+8b\equiv2\pmod4.
\]
Conversely, a principal basis vector in a degree $h\equiv2\pmod4$ has second exponent $m=4b$ and first exponent $h-2m\equiv2\pmod4$, so it has the form $v_{4a+2,4b}$.  The extra endpoint in degree $h\equiv6\pmod8$ is $v_{0,3+4a}$.  Thus their classes give the degreewise cohit bases.

Let $L$ be the $\A(1)$-submodule generated by their union. We prove $M_h\subseteq L$ by induction on $h\geq0$. In degree zero, $M_0$ is spanned by $v_{0,0}$, which belongs to $\mathcal G_{0,1}$. Suppose that $h>0$ and that every smaller nonnegative degree lies in $L$. For $x\in M_h$, its quotient class is a linear combination of the displayed generators of degree $h$. Let $g$ be the corresponding combination in $M_h$. By \eqref{eq:A1-hit-space}, there exist $y\in M_{h-1}$ and $z\in M_{h-2}$ such that
\[
x-g=\Sq^1(y)+\Sq^2(z).
\]
Any negative-degree source is zero by convention. All other source vectors belong to $L$ by induction, so $x\in L$. This proves that the union generates $M$.

To prove minimality, suppose one generator were generated by the others. Passing to $M/\A(1)^+M$ annihilates every positive-degree operation and expresses its quotient class as an $\F_2$-linear combination of the other classes. This contradicts the basis assertion. Hence no generator can be removed.

The generating function of $\{v_{4a,2+4b}\}$ is
\[
\sum_{a,b\geq0}t^{4a+4+8b}
=
\frac{t^4}{(1-t^4)(1-t^8)}.
\]
The generating function of $\{v_{4a+2,4b}\}$ is
\[
\sum_{a,b\geq0}t^{4a+2+8b}
=
\frac{t^2}{(1-t^4)(1-t^8)}.
\]
The two endpoint families contribute
\[
\sum_{a\geq0}t^{6+8a}=\frac{t^6}{1-t^8},
\qquad
\sum_{a\geq0}t^{8a}=\frac{1}{1-t^8}.
\]
Adding the four series gives \eqref{eq:hilbert-series}. Multiplying by $t^2$ changes internal degree $h$ to cohomological degree $h+2$.
\end{proof}

The following summands are defined using the Stiefel--Whitney monomial basis, not by choosing representatives of a minimal generating set. In this section, all decompositions and endomorphisms are in the category of graded $\A(1)$-modules with degree-preserving maps; no topological wedge decomposition is asserted.

Define
\begin{align*}
M^{[0,1]}
&=\Span_{\F_2}\{v_{n,m}:m\equiv0\text{ or }1\pmod4\},\\
M^{[2,3]}
&=\Span_{\F_2}\{v_{n,m}:m\equiv2\text{ or }3\pmod4\}.
\end{align*}

These are the residue subspaces underlying the decomposition stated in \cite[Theorem~4.3]{Shih2012}. Their stability under the two algebra generators gives a direct verification of that decomposition.

\begin{theorem}\label{thm:A1-splitting}
Let $M^{[0,1]}$ and $M^{[2,3]}$ be the monomial spans defined above. Then both are $\A(1)$-submodules, and
\[
M=M^{[0,1]}\oplus M^{[2,3]}
\]
as graded $\A(1)$-modules.
\end{theorem}

\begin{proof}
The two subspaces are spanned by disjoint subsets of the monomial basis $\{v_{n,m}\}$, and every nonnegative integer $m$ has exactly one of the four residues modulo four.  Therefore
\[
M=M^{[0,1]}\oplus M^{[2,3]}
\]
as graded vector spaces.  It remains to prove stability under the generators $\Sq^1$ and $\Sq^2$.

By \eqref{eq:sq1-nm}, the second index of every nonzero $\Sq^1$ image is unchanged. Hence $\Sq^1$ preserves each subspace.

The first term in \eqref{eq:sq2-nm} preserves $m$. The second term has index $m+1$ and coefficient $m+1$, so it is nonzero exactly when $m$ is even.  If $m\equiv0\pmod4$, the possible increase sends $m$ to a value congruent to $1$ modulo four.  If $m\equiv2\pmod4$, it sends $m$ to a value congruent to $3$ modulo four.  If $m$ is odd, the second term vanishes.  Thus $\Sq^2$ preserves the pair of residues $\{0,1\}$ and the pair $\{2,3\}$.  Since $\A(1)$ is generated by $\Sq^1$ and $\Sq^2$, both subspaces are $\A(1)$-submodules.  The vector-space direct sum is therefore a direct sum of $\A(1)$-modules.
\end{proof}

\begin{remark}\label{rem:idempotent}
Let $p$ be the coordinate projection onto $M^{[0,1]}$ along $M^{[2,3]}$. Theorem~\ref{thm:A1-splitting} gives $p\Sq^i=\Sq^ip$ for $i=1,2$, so $p\in\End_{\A(1)}(M)$ and $p^2=p$. It is neither zero nor the identity, because $p(v_{0,0})=v_{0,0}$ and $p(v_{0,2})=0$.
\end{remark}

The direct sum in Theorem~\ref{thm:A1-splitting} separates the generators of Proposition~\ref{prop:global-generators} according to their second indices, giving the following consequence.

\begin{corollary}\label{cor:summand-generators}
The sets $\mathcal G_{0,1}$ and $\mathcal G_{2,3}$ of Proposition~\ref{prop:global-generators} are minimal homogeneous generating sets for $M^{[0,1]}$ and $M^{[2,3]}$, respectively.
\end{corollary}

\begin{proof}
Since every element of $\A(1)^+$ preserves the summands in Theorem~\ref{thm:A1-splitting},
\[
\A(1)^+M=\A(1)^+M^{[0,1]}\oplus\A(1)^+M^{[2,3]}.
\]
Taking quotients therefore gives the corresponding direct-sum decomposition of cohits.  The global basis in Proposition~\ref{prop:global-generators} separates according to the residue of $m$ modulo four, with $\mathcal G_{0,1}$ lying in $M^{[0,1]}$ and $\mathcal G_{2,3}$ lying in $M^{[2,3]}$.  Each subset is therefore a basis of the cohit of its summand. Both summands are zero in negative internal degrees, so the induction and minimality argument in Proposition~\ref{prop:global-generators} apply to them as well.
\end{proof}

\begin{remark}[Extensions to $\A(2)$ and $\A(3)$]\label{rem:higher-subalgebras}
The certification method of Proposition~\ref{prop:certificate-principle} applies without change to $\A(2)$ and $\A(3)$, where the hit matrices acquire the columns of $\Sq^4$ (at most three terms by Theorem~\ref{thm:general-square}) and $\Sq^8$ (at most five terms). Extending our exactness results to those cases requires uniform row-column minor selections and dual annihilators across all degrees; the candidate generating sets in \cite[Sections~4.3--4.4]{Shih2012} are not established here. Likewise, since a graded module splits if and only if its endomorphism ring contains a nontrivial degree-preserving idempotent, settling the non-splitting conjectures in \cite[Conjectures~4.6--4.7]{Shih2012} requires ruling out all such idempotents, not only overlaps of cyclic submodules from a chosen generating list.
\end{remark}


\begin{thebibliography}{99}

\bibitem{Ault2014}
S.~V. Ault,
\emph{Bott periodicity in the hit problem},
Math. Proc. Cambridge Philos. Soc. \textbf{156} (2014), 545--554.
DOI: \doilink{10.1017/S0305004114000085}.

\bibitem{Fine1947}
N.~J. Fine,
\emph{Binomial coefficients modulo a prime},
Amer. Math. Monthly \textbf{54} (1947), 589--592.
DOI: \doilink{10.2307/2304500}.

\bibitem{Janfada2008}
A.~S. Janfada,
\emph{A criterion for a monomial in $P(3)$ to be hit},
Math. Proc. Cambridge Philos. Soc. \textbf{145} (2008), 587--599.
DOI: \doilink{10.1017/S030500410800162X}.

\bibitem{Janfada2009}
A.~S. Janfada,
\emph{A note on the unstability conditions of the Steenrod squares on the polynomial algebra},
J. Korean Math. Soc. \textbf{46} (2009), 907--918.
DOI: \doilink{10.4134/JKMS.2009.46.5.907}.

\bibitem{Janfada2011}
A.~S. Janfada,
\emph{On a conjecture on the symmetric hit problem},
Rend. Circ. Mat. Palermo (2) \textbf{60} (2011), 403--408.
DOI: \doilink{10.1007/s12215-011-0062-2}.

\bibitem{Janfada2014}
A.~S. Janfada,
\emph{Criteria for a symmetrized monomial in $B(3)$ to be non-hit},
Commun. Korean Math. Soc. \textbf{29} (2014), 463--478.
DOI: \doilink{10.4134/CKMS.2014.29.3.463}.

\bibitem{JanfadaWood2002}
A.~S. Janfada and R.~M.~W. Wood,
\emph{The hit problem for symmetric polynomials over the Steenrod algebra},
Math. Proc. Cambridge Philos. Soc. \textbf{133} (2002), 295--303.
DOI: \doilink{10.1017/S0305004102006059}.

\bibitem{JanfadaWood2003}
A.~S. Janfada and R.~M.~W. Wood,
\emph{Generating $H^*(BO(3),\F_2)$ as a module over the Steenrod algebra},
Math. Proc. Cambridge Philos. Soc. \textbf{134} (2003), 239--258.
DOI: \doilink{10.1017/S0305004102006394}.

\bibitem{Kameko1998}
M.~Kameko,
\emph{Generators of the cohomology of $BV_3$},
J. Math. Kyoto Univ. \textbf{38} (1998), 587--593.
DOI: \doilink{10.1215/kjm/1250518069}.

\bibitem{MilnorStasheff1974}
J.~W. Milnor and J.~D. Stasheff,
\emph{Characteristic classes},
Annals of Mathematics Studies, no.~76, Princeton University Press, Princeton, NJ, 1974.
DOI: \doilink{10.1515/9781400881826}.

\bibitem{PengelleyWilliams2011}
D.~J. Pengelley and F.~Williams,
\emph{A new action of the Kudo--Araki--May algebra on the dual of the symmetric algebras, with applications to the hit problem},
Algebr. Geom. Topol. \textbf{11} (2011), 1767--1780.
DOI: \doilink{10.2140/agt.2011.11.1767}.

\bibitem{PengelleyWilliams2015}
D.~J. Pengelley and F.~Williams,
\emph{Sparseness for the symmetric hit problem at all primes},
Math. Proc. Cambridge Philos. Soc. \textbf{158} (2015), 269--274.
DOI: \doilink{10.1017/S0305004114000668}.

\bibitem{Peterson1989}
F.~P. Peterson,
\emph{$\A$-generators for certain polynomial algebras},
Math. Proc. Cambridge Philos. Soc. \textbf{105} (1989), 311--312.
DOI: \doilink{10.1017/S0305004100067803}.

\bibitem{Phuc2026Minors}
\DJ. V. Ph\'uc and P. V. Danchev,
\emph{Two-Sided Dimension Bounds for the Peterson Hit Problem via Projections and Matrix Minors}, 
arXiv:2608.02623v5 [math.AT] (2026), 26 pages (submitted for publication).
DOI: \doilink{10.48550/arXiv.2608.02623}.

\bibitem{Shih2012}
M.~Shih,
\emph{On the splitting of $MO(2)$ over the Steenrod algebra},
unpublished Research Science Institute report, Massachusetts Institute of Technology, Cambridge, MA, 2012; available at \url{https://math.mit.edu/documents/rsi/2012Shih.pdf}.

\bibitem{Singer1991}
W.~M. Singer,
\emph{On the action of Steenrod squares on polynomial algebras},
Proc. Amer. Math. Soc. \textbf{111} (1991), 577--583.
DOI: \doilink{10.1090/S0002-9939-1991-1045150-9}.

\bibitem{Singer2008}
W.~M. Singer,
\emph{Rings of symmetric functions as modules over the Steenrod algebra},
Algebr. Geom. Topol. \textbf{8} (2008), 541--562.
DOI: \doilink{10.2140/agt.2008.8.541}.

\bibitem{Thom1952}
R.~Thom,
\emph{Espaces fibr\'es en sph\`eres et carr\'es de Steenrod},
Ann. Sci. \'Ec. Norm. Sup\'er. (3) \textbf{69} (1952), 109--182.
DOI: \doilink{10.24033/asens.998}.

\bibitem{Thom1954}
R.~Thom,
\emph{Quelques propri\'et\'es globales des vari\'et\'es diff\'erentiables},
Comment. Math. Helv. \textbf{28} (1954), 17--86.
DOI: \doilink{10.1007/BF02566923}.

\bibitem{WalkerWood1}
G.~Walker and R.~M.~W. Wood,
\emph{Polynomials and the mod $2$ Steenrod algebra. Vol.~1: The Peterson hit problem},
London Math. Soc. Lecture Note Ser., vol.~441, Cambridge University Press, Cambridge, 2018.
DOI: \doilink{10.1017/9781108333368}.

\bibitem{WalkerWood2}
G.~Walker and R.~M.~W. Wood,
\emph{Polynomials and the mod $2$ Steenrod algebra. Vol.~2: Representations of $GL(n,\F_2)$},
London Math. Soc. Lecture Note Ser., vol.~442, Cambridge University Press, Cambridge, 2018.
DOI: \doilink{10.1017/9781108304092}.

\bibitem{Wood1989}
R.~M.~W. Wood,
\emph{Steenrod squares of polynomials and the Peterson conjecture},
Math. Proc. Cambridge Philos. Soc. \textbf{105} (1989), 307--309.
DOI: \doilink{10.1017/S0305004100067797}.

\end{thebibliography}
\end{document}